\documentclass[12pt]{article}

\usepackage{a4wide}
\usepackage{graphicx}
\usepackage{amsmath}
\usepackage{amssymb}
\usepackage{amsthm}
\usepackage{enumerate}

\newtheorem{theorem}{Theorem}

\newtheorem{claim}[theorem]{Claim}

\DeclareMathOperator{\dist}{\textit{d}}
\DeclareMathOperator{\diam}{diam}

\begin{document}
\title{Edge growth in graph powers}
\author{A. Pokrovskiy\\ London School of Economics and Political Sciences \\ a.pokrovskiy@lse.ac.uk}
\maketitle

\begin{abstract}
For a graph $G$, its $r$th power $G^r$ has the same vertex set as $G$, and has an edge between any two vertices within distance $r$ of each other in $G$. We give a lower bound for the number of edges in the $r$th power of $G$ in terms of the order of $G$ and the minimal degree of $G$.  As a corollary we determine how small the ratio $e(G^r)/e(G)$ can be for regular graphs of diameter at least $r$. 
\end{abstract}

\section{Introduction}
We will consider both graphs that may have loops and graphs in which loops are explicitly forbidden.  Loopless graphs will be denoted by Roman italic letters, such as ``$G$", while graphs with loops allowed will be denoted by curly letters, such as ``$\mathcal{G}$".  For two vertices~$x$ and $y$ (possibly $x = y$) we only allow one edge between $x$ and $y$.
The $r$th power of~$G$, denoted $G^r$, is the graph with vertex set $V(G)$, and $xy$ an edge whenever $x$ and $y$ are within distance $r$ of each other.  
The \emph{diameter} of a connected graph is the smallest $r$ for which $G^r$ is complete.    For all standard notation we refer to \cite{Diestel}.

For a connected graph of diameter at least $r$, one would expect $G^r$ to have substantially more edges than $G$.  In this note we examine how small the ratio $e(G^r)/e(G)$ can be, focusing primarily on the case when $G$ is a regular graph.

The motivation for studying this comes from a corollary  of the Cauchy-Davenport Theorem from additive number theory which we will now state.  The Cayley graph of a subset $A\subseteq \mathbb{Z}_p$ is constructed on the vertex set $\mathbb{Z}_p$. For two distinct vertices $x, y \in \mathbb{Z}_p$, we define $xy$ to be an edge whenever $x -  y \in A$ or $y -  x \in A$.  The following is a consequence of the Cauchy-Davenport Theorem (usually stated in the language of additive number theory).

\begin{theorem} [Cauchy, Davenport, \cite{Cauchy, Davenport}]
Let $p$ be a prime, $G$ the Cayley graph of a set $A \subseteq \mathbb{Z}_p$, and $r$ an integer such that $r < \diam(G)$.  Then we have
\begin{equation} \label{eq:cauchy}
\frac{e(G^r)}{e(G)}\geq r.
\end{equation}
\end{theorem}

One could ask whether inequalities similar to (\ref{eq:cauchy}) hold for more general families of graphs.  Motivated by the fact that Cayley graphs are regular, Hegarty asked this question for regular graphs and proved the following theorem.

\begin{theorem} [Hegarty, \cite{Hegarty}] \label{thm:hegarty}
Let $G$ be a regular, connected graph, with $\diam(G)\geq 3$.  Then we have
\begin{equation}\label{eq:hegarty}
\frac{e(G^3)}{e(G)}\geq 1+ \epsilon.
\end{equation}
Where $\epsilon \approx  0.087.$
\end{theorem}

The constant $\epsilon$ has since been improved to $\frac{1}{6}$ by the author \cite{Pokrovskiy} and to $\frac{3}{4}$ by DeVos and Thomass\'e \cite{Thomasse}.  The value $\epsilon = \frac{3}{4}$ is optimal in the sense that there exists a sequence of regular graphs of diameter greater than $3$, $G_m$, satisfying $\frac{e(G_m^3)}{e(G_m)}\to \frac{7}{4}$ as $m \to \infty$ \cite{Thomasse}.  It is natural to ask what happens for other powers of $G$.

For $G^2$, Hegarty showed that no inequality similar to (\ref{eq:hegarty}) with $\epsilon>0$  can hold for regular graphs in general, by exhibiting a sequence of regular, connected graphs of diameter greater than $2$, $G_m$, satisfying $\frac{e(G_m^2)}{e(G_m)}\to 1$ as $m \to \infty$ \cite{Hegarty}.  Goff \cite{Goff} studied the $2$nd power of regular graphs further and showed that for any $d$-regular connected graph $G$ such that $\diam(G)>2$, we have $\frac{e(G^2)}{e(G)}\geq 1+\frac{3}{2 d}-o\left(\frac{1}{d}\right)$.  For general $d$-regular connected graphs $G$ with $\diam(G)>2$, the $\frac{3}{2d}$ term in this result cannot be replaced with $\frac{\lambda}{ d}$ for any $\lambda>\frac{3}{2}$.  However it is shown in \cite{Goff} that with the exception of two families of exceptional graphs, we have  $\frac{e(G^2)}{e(G)}\geq 1+\frac{2}{d}-o\left(\frac{1}{d}\right)$ for all $d$-regular connected graphs with $\diam(G)>2$.

In this note we consider all $r\geq 4$ and determine how small $\frac{e(G^r)}{e(G)}$ can be for $G$ a regular, connected graph of diameter at least $r$.  We prove the following theorem. 

\begin{theorem} \label{thm:regular}
Let $G$ be a connected, regular graph, and $r$ a positive integer such that $\diam(G) \geq r$.  
\begin {itemize}
 \item If $r\equiv 0 \pmod{3}$, then we have
$$\frac{e(G^r)}{e(G)}\geq \frac{r+3}{3}-\frac{3}{2(r+3)}  .$$
 \item If $r\not \equiv 0 \pmod{3}$, then we have
$$\frac{e(G^r)}{e(G)}\geq \left\lceil \frac{r}{3} \right\rceil .$$
\end {itemize}
\end{theorem}

The case $r=3$ of Theorem~\ref{thm:regular} is due to DeVos and Thomass\'e \cite{Thomasse}, and will not be proved here.  Theorem \ref{thm:regular} gives a lower bound on the ratio $\frac{e(G^r)}{e(G)}$ for regular graphs. The bounds on $\frac{e(G^r)}{e(G)}$ in Theorem~\ref{thm:regular} are optimal in the following sense.  For each $r$, there exists a sequence of regular, connected graphs of diameter at least $r$, $G_m$, such that $\frac{e(G_m ^r)}{e(G_m)}$ tends to the bound given by Theorem~\ref{thm:regular} as $m$ tends to infinity.  We refer to Figure~\ref{examples} for a diagram of the sequences that we construct.
\begin{figure}
  \centering
     \includegraphics[width=0.7\textwidth]{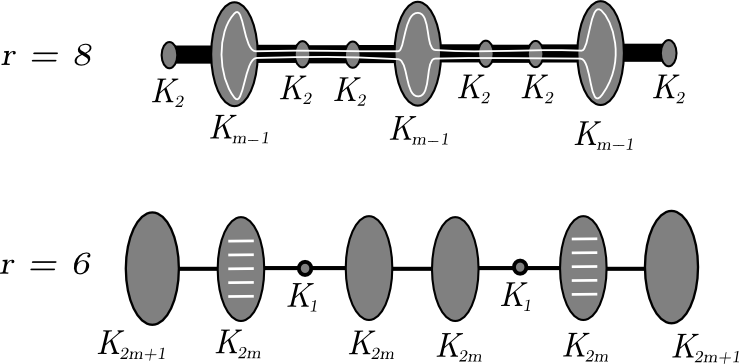}
  \caption{ Graphs showing the cases ``$r = 8$'' and ``$r=6$'' of Theorem~\ref{thm:regular} to be optimal.  The grey ovals represent complete graphs of specified order.  The black lines between the sets represent all the edges being present between them.  The white cycle in the ``$r=8$'' case represents the removal of a single cycle passing through all the vertices in the sets it intersects. The white matchings in the ``$r=6$'' case represent a perfect matching being removed from the specified sets.\label{examples}}
\end{figure}

To see this for $r \not\equiv 0 \pmod{3}$, we  construct the following sequence of graphs $G_m$.  Take disjoint sets of vertices $N_0,...,N_r$,  with $|N_i|=m - 1$ if $i\equiv 1 \pmod 3$ and $|N_i| = 2$ otherwise.  Add all the edges between $N_i$ and $N_{i+1}$ for $i = 0$, $1$, $\dots$, $r-1$.  Add all the edges within $N_i$ for all $i$.  Remove a cycle passing through all the vertices in $N_1 \cup ... \cup N_{r-1}$.  It is easy to see that $G_m$ is $m$-regular and has diameter $r$.  If $r\equiv 1 \pmod 3$ then $|G_m|=\frac{1}{3}(rm+2m+3r)$ will hold.  Since $G_m$ is $m$-regular, we have $e(G_m)=\frac{1}{6}(rm+2m+3r)m$.  Since $G_m ^r$ is complete, we have $e(G_m ^r)= \frac{1}{18}(rm+2m+3r)(rm+2m+3r-1)$.  This implies that $\frac{e(G_m ^r)}{e(G_m)} \to \left\lceil \frac{r}{3} \right\rceil$ as $m \to \infty$.  A similar calculation can be used to show that the same limit holds when $r\equiv 2 \pmod 3$.

For $r \equiv 0 \pmod{3}$,  we construct the following sequence of graphs $H_m$ to show that Theorem~\ref{thm:regular} is optimal.
Take disjoint sets of vertices $N_0,...,N_{r+1}$.  Let $|N_0|=|N_{r+1}|=2m+1$, $|N_i|= 1$ if $i \equiv 2 \pmod{3}$, and $|N_i|= 2m$ otherwise.  Add all the edges between $N_i$ and $N_{i+1}$ for $i = 0$, $1$, $\dots$, $r$.  Add all the edges within $N_i$ for all $i$.
  Delete a perfect matching from each of the sets $N_2$ and $N_{r}$.  This will ensure that $H_m$ is $4m$-regular and has diameter $r+1$.  Note that $|H_m|=\frac{1}{3}(4rm+r+12m+6)$, and so we have $e(H_m)=\frac{1}{6}(4rm+r+12m+6)4m$.  The only edges missing from $H_m ^r$ will be between $N_0$ and $N_{r+1}$, so we have $e(H_m ^r)=\frac{1}{18}(4rm+r+12m+6)(4rm+r+12m+5)-(2m+1)^2$.  This implies that  $\frac{e(H_m ^r)}{e(H_m)} \to \frac{r+3}{3}-\frac{3}{2(r+3)}$ as $m \to \infty$.
This construction is a generalization of one from \cite{Thomasse}.

All the examples constructed above have their diameter close to $r$.  If a graph $G$ has diameter larger than $r$, it seems that the bounds of Theorem \ref{thm:regular} can be improved.  Some results in this direction have been obtained DeVos, McDonald and Scheide \cite{DeVos}.  

The requirement of $G$ being regular in the above theorems is quite restrictive.  Following~\cite{Thomasse}, we will instead assume that $G$ has minimum degree $\delta(G)$, and give the following bound on $e(G^r)$ in terms of $|G|$ and $\delta(G)$.

\begin{theorem}\label{thm:noloops}
Let ${G}$ be a connected graph, and $r$ a positive integer such that $\diam(G)\geq r$.
\begin {itemize}
 \item If $r\equiv 0 \pmod{3}$, then we have
$$e(G^r)\geq \left(\frac{r+3}{6}-\frac{3}{4(r+3)} \right) \delta(G) |G|.$$
 \item If $r\not \equiv 0 \pmod{3}$, then we have
$$e(G^r)\geq \frac{1}{2}\left\lceil \frac{r}{3} \right\rceil \delta(G) |G|.$$
\end {itemize}
\end{theorem}

The case $r=3$ of Theorem~\ref{thm:noloops} is due to DeVos and Thomass\'e \cite{Thomasse}, and will not be proved here.  Theorem~\ref{thm:noloops} easily implies Theorem~\ref{thm:regular}.

\section{Proof of Theorem~\ref{thm:noloops}}

We will prove a version of Theorem~\ref{thm:noloops} for graphs which may contain loops since in that setting the proof seems more natural.

The neighbourhood of a vertex $x$, $N(x)$, is defined as the set of vertices adjacent to~$x$. (If there is a loop at $x$, then $N(x)$ will contain $x$ itself.)  The degree of $x$ is $|N(x)|$.    For graphs with loops allowed, $\mathcal{G}^r$ is defined identically to how it was defined for loopless graphs.  Note that if $\mathcal{G}$ is a graph with loops allowed, then $\mathcal{G}^r$ always has a loop at each vertex. 
For two sets of vertices $X$ and $Y$, let $d(X,Y)$ denote the length of a shortest path between a vertex in $X$ and a vertex in $Y$.  If $X$ is a set of vertices, let $N^r(X)$ be the set of vertices at distance at most $r$ from $X$.  We abbreviate $N^r(\{x\})$ as $N^r(x)$ and $d(\{x\},\{y\})$ as $d(x,y)$.

We prove the following theorem, and then deduce Theorem~\ref{thm:noloops} as a corollary.  Several ideas in the proof of Theorem~\ref{thm:loops} are taken from~\cite{Thomasse}.  In particular, Claims \ref{lem:vertexbound} and \ref{lem:insufficientbound} are analogues of claims proved in~\cite{Thomasse}.
\begin{theorem}\label{thm:loops}
Let $\mathcal{G}$ be a connected graph, and $r$ a positive integer such that $r\geq 6$ and $\diam(G)\geq r$.
\begin {itemize}
 \item If $r\equiv 0 \pmod{3}$, then we have
$$e(\mathcal{G}^r)\geq \left(\frac{r+3}{6}-\frac{3}{4(r+3)} \right)\delta(\mathcal{G}) |\mathcal{G}| + \frac{1}{2}|\mathcal{G}|.$$
 \item If $r\not \equiv 0 \pmod{3}$, then we have
$$e(\mathcal{G}^r)\geq \frac{1}{2}\left\lceil \frac{r}{3} \right\rceil \delta(\mathcal{G}) |\mathcal{G}| + \frac{1}{2}|\mathcal{G}|.$$
\end {itemize}
\end{theorem}

\begin{proof}
For convenience, we will set $\delta = \delta(\mathcal{G})$.
If $P$ is a path between two vertices $x$ and $y$, we say that $P$ is a \emph{geodesic} if the length of $P$ is $\dist(x, y)$.   The notion of a geodesic was used in \cite{Thomasse}, and is useful because the neighbourhood of a geodesic must be quite large.  This is quantified in the following claim.

\begin{claim}\label{cla:geodesic}
Let $P$ be a length $k$ geodesic.  Then $|N(P)|\geq \left(\left\lfloor \frac{k}{3} \right\rfloor +1 \right)\delta$ holds.
\end{claim}
\begin{proof}
If $x_0, \ x_1, \dots,  x_{k}$ are the vertices of $P$ (in the order in which they occur along the path), then $N(x_0), \ N(x_3), \dots, N(x_{3\left\lfloor \frac{k}{3} \right\rfloor })$ are all disjoint, contained in $N(P)$, and of order at least $\delta$.  This implies the result.
\end{proof}

We now prove the case ``$r\not\equiv 0 \pmod{3}$'' of the theorem.

The diameter of $\mathcal{G}$ is at least $r$, so $\mathcal{G}$ contains a length $r$ geodesic, $P$. Claim \ref{cla:geodesic} implies that the following holds: 
\begin{equation}\label{eqn:thmeasy}
|\mathcal{G}|\geq |N(P)| \geq  \left(\left\lfloor \frac{r}{3} \right\rfloor +1 \right)\delta \geq \left\lceil \frac{r}{3} \right\rceil \delta.
\end{equation}

 Note that $\mathcal{G}^r$ contains a loop at every vertex, so we have $e(\mathcal{G}^r)=\sum_{v \in V(\mathcal{G})} \left( \frac{1}{2}|N^r (v)|+ \frac{1}{2} \right)$.  Thus to prove Theorem \ref{thm:loops} it is sufficent to exhibit $\left\lceil \frac{r}{3} \right\rceil \delta$ elements of $N^r(v)$ for each vertex $v \in V(G)$. 

Let $v$ be a vertex in $G$.  Suppose that there exists a length $r-1$ geodesic $P_v$ starting from $v$.  Then $N(P_v)$ is contained in $N^r(v)$, giving $$|N^r(v)|\geq |N(P_v)| \geq  \left(\left\lfloor \frac{r-1}{3} \right\rfloor +1 \right)\delta = \left\lceil \frac{r}{3} \right\rceil \delta.$$  
The second inequality is an application of Claim \ref{cla:geodesic}.

Suppose that all the vertices in $\mathcal{G}$ are within distance $r-1$ of $v$. In this case we have $N^r(v)=V(\mathcal{G})$, which is of order at least  $\left\lceil \frac{r}{3} \right\rceil \delta$ by (\ref{eqn:thmeasy}).  This completes the proof of the case ``$r\not\equiv 0 \pmod{3}$'' of the theorem.

\

For the rest of the proof fix $r$ such that $r\equiv 0 \pmod{3}$ and $r \geq 6$.

If $v$ is a vertex of $\mathcal{G}$, we say that $v$ is \emph{sufficient} if $|N^r(v)|\geq \left(\frac{r}{3}+1\right)\delta$.  Otherwise we say that $v$ is \emph{insufficient}.

The following is a useful property of insufficient vertices.
\begin{claim}\label{cla:longpath}
Let $v$ be an insufficient vertex.  Then there is some vertex at distance $r+1$ from~$v$.
\end{claim}
\begin{proof}
Since $\diam(\mathcal{G})\geq r$, Claim \ref{cla:geodesic} implies that $|\mathcal{G}| \geq \left(\frac{r}{3}+1\right)\delta$. By assumption $N^r(v)<\left(\frac{r}{3}+1\right)\delta$, so $v$ cannot be within distance $r$ from all the vertices in the graph.  
\end{proof}

The following three claims will allow us to bound the number of insufficient vertices in~$\mathcal{G}$.
\begin{claim}\label{lem:nice}
If $2<d(x,y)<r$ holds for  $x,y \in V(\mathcal{G})$, then either $x$ or $y$ is sufficient.
\end{claim}
\begin{proof}
Suppose that $x$ is insufficient. By Claim \ref{cla:longpath}, we can find a length $r$ geodesic starting from $x$ with vertex sequence $x,$ $x_1,$ $x_2, \dots, x_r$.

Suppose that $N(y)\cap N(x_i) \neq \emptyset$ for some $i$ with $3\leq i \leq r-3$. In this case $N(x),$ $N(x_3),$ $N(x_6),\dots , N(x_r)$ are all contained in $N^r(y)$.  There are $\frac{r}{3}+1$ of these, they are all disjoint (since $x,$ $x_1,$ $x_2, \dots, x_r$ form a geodesic), and are of order at least $\delta$.  Hence $y$ is sufficient.

Otherwise $N(y)\cap N(x_i) = \emptyset $  for all  $3\leq i \leq r-3$.  In this case $N(x),$ $N(y),$ $N(x_3),$ $N(x_6),\dots, N(x_{r-3})$ are all disjoint and contained in $N^r(x)$.  This contradicts our initial assumption that $x$ is insufficient.
\end{proof}

\begin{claim}\label{lem:notnice}
Let $x$ and $y$ be two vertices in $\mathcal{G}$ such that $d(x,y)= r$ or $d(x,y)=r+1$.  If there exists a vertex $z \in \mathcal{G}$ such that $d(z,x), \ d(z,y) \geq r-1$, then either $x$ or $y$ is sufficient.
\end{claim}

\begin{proof}
Choose any $z$ in $N^{r-1}(\{x,y\}) \setminus N^{r-2}(\{x,y\}$. This set is nonempty by the second assumption of the claim.  We will have $d(z,x), d(z,y) \geq r-1$ and either $d(z,x)=r-1$ or $d(z,y) = r-1$.  Without loss of generality assume that $d(z,x)=r-1$ and $d(z,y)\geq r-1$.

We will show that $x$ is sufficient.
Let $x,\ x_1, \dots, x_{d(x,y)-1},\ y$ be a geodesic between $x$ and $y$.
For $i = 1$, $\dots$, $d(x,y)-1$, the triangle inequality implies that
\begin{align}
 d(x_i,z) &\geq d(x,z) -d(x,x_i)=d(x,z)-i \label{eqn:triangle1}, \\
 d(x_i,z) &\geq d(y,z) -d(y,x_i)=d(y,z)-d(x,y)+i \label{eqn:triangle2}.
\end{align}
Averaging (\ref{eqn:triangle1}) and (\ref{eqn:triangle2}), and use the inequalities $d(z,x), d(z,y) \geq r-1$ and $d(x,y)\leq r+1$ gives 
\begin{equation}\label{eqn:average}
d(x_i,z)\geq \frac{r-3}{2}.
\end{equation}

If $r \geq 9$, then (\ref{eqn:average}) implies that $d(x_i,z)\geq 3$ for all $i$.  Hence  $N(x),$ $N(z),$ $N(x_3),$ $N(x_6),\dots, N(x_{r-3})$ are all disjoint and contained in $N^r(x)$.  Hence $x$ is sufficient.

If $r=6$, then (\ref{eqn:triangle1}) and (\ref{eqn:triangle2}) imply that $d(x_i,z)\geq 3$ for all $x_i$ except possibly $x_3$ or $x_4$.  In this case $N(z),\ N(x_2)$ and $N(x_5)$ are all disjoint and contained in $N^6(x)$.  Hence $x$ is sufficient.
\end{proof}

\begin{claim}\label{cor:cor}
If $d(x,y)= r$ holds for  $x,y \in V(\mathcal{G})$, then either $x$ or $y$ is sufficient.
\end{claim}
\begin{proof}
Suppose that $x$ and $y$ are insufficient.
By Claim \ref{cla:longpath} there exists $z \in V(\mathcal{G})$ such that $d(x,z)=r+1$.  Let $x,\ x_1, \dots, x_{r-1}, y$ be a geodesic between $x$ and $y$.  Since $x$ and $y$ are insufficient, Claim \ref{lem:notnice} implies that we have $d(z,y) < r-1$. Note that $d(x,z)=r+1$ implies that $N(z) \cap N(x_i)= \emptyset$ for all $i\leq r-2 $.  Thus $N(z),\ N(x_1),\ N(x_4), \dots, N(x_{r-2})$ are all disjoint and  contained in $N^r(y)$.  This contradicts our assumption that $y$ is insufficient.
\end{proof}

Let $X$ be the set of insufficient vertices in $\mathcal{G}$.  We define an equivalence relation ``$\sim$" on~$X$ by letting $x \sim y$ if $d(x,y)\leq 2$.  For $r \geq 6$, Claim \ref{lem:nice} implies that this is an equivalence relation.   Let $X_1, \dots, X_l$ be the equivalence classes  of $\sim$.

The following claim gives a lower bound on the order of $\mathcal{G}$. 
\begin{claim}\label{lem:vertexbound}
$|\mathcal{G}| \geq \left( \frac{r+3}{6} \right) \delta l $
\end{claim}
\begin{proof}
Claims \ref{lem:nice} and \ref{cor:cor} imply that $d(X_i, X_j)\geq r+1$ for all $i\ne j$.  If $d(X_i, X_j)=r+1$ for some $i$ and $j$, then Claim \ref{lem:notnice} implies that we have $d(X_i, z) < r-1$ or $d(X_j, z) < r-1$ for all $z \in V(\mathcal{G})$.  Then, Claim \ref{lem:nice} implies that all the vertices outside of $X_i$ and $X_j$ are sufficient.  This gives us two cases to consider:
\begin{enumerate}[(i)]
\item $d(X_i, X_j)\geq r+2$ for all $i\neq j$.
\item  $d(X_1, X_2)=r+1$.
\end{enumerate}

Suppose that (i) holds (this includes the case when $l=1$).  For each $i$, choose $x_i$ to be a vertex in $X_i$.  
Note that $N^{\left\lfloor\frac{r}{2}\right\rfloor} (x_i)$ contains a length $\left\lfloor \frac{r}{2}\right\rfloor$ geodesic, $P_i$.  Using Claim \ref{cla:geodesic} gives
$$\left| N^{\left\lfloor\frac{r}{2}\right\rfloor +1} (X_i) \right|\geq |N(P_i)| \geq \left( \left\lfloor \frac{1}{3}\left\lfloor \frac{r}{2} \right\rfloor \right\rfloor +1 \right)\delta \geq \left(\frac{r+3}{6}\right)\delta .$$
For the last inequality we are using the fact that $r \equiv 0 \pmod{3}$.   Note that  (i) implies that $N^{\left\lfloor\frac{r}{2}\right\rfloor +1} (X_i) \cap N^{\left\lfloor\frac{r}{2}\right\rfloor +1} (X_j) = \emptyset$ for all $i, j$.  This implies that the following holds:
$$|V(\mathcal{G})|\geq  \sum_{i=1} ^l \left| N^{\left\lfloor\frac{r}{2}\right\rfloor +1} (X_i) \right| \geq \left(\frac{r+3}{6}\right)\delta l .$$

Suppose that (ii) holds.
Using Claim \ref{cla:geodesic} we obtain $$|V(\mathcal{G})|\geq \left( \frac{r}{3} +1 \right)\delta=\left(\frac{r+3}{6}\right)\delta l .$$
\end{proof}

When $x$ is insufficient, the following claim gives a lower bound on the order of $N^r(x)$.  
\begin{claim}\label{lem:insufficientbound}
Suppose that  $x$ is an insufficient vertex in the equivalence class $X_i$.  Then, $\left|N^r(x) \right| \geq \left|X_i \right| + \frac{r}{3}\delta$ holds.
\end{claim}
\begin{proof}
By Claim \ref{cla:longpath}, we can choose a length $r$ geodesic from $x$.  Let $x, x_1, \dots, x_r$ be the vertices of this geodesic.  Suppose that $X_i \cap N(x_j)$ is nonempty for some $x_j$.  Choose $y\in X_i \cap N(x_j)$. Clearly $j\leq 1$ must hold, since otherwise $N(x),\ N(x_3), N(x_6), \dots, N(x_r)$ would all be contained in $N^r(y)$, contradicting that $y$ is insufficient (since $y \in X_i$).

Hence $X_i,\ N(x_2),\ N(x_5), \dots, N(x_{r-1})$ are all disjoint and contained in $N^r(x)$ proving the claim.
\end{proof}

Claims \ref{lem:vertexbound} and \ref{lem:insufficientbound} are all that is needed to prove Theorem \ref{thm:loops}, as follows

\begin{align*}
2e(\mathcal{G}^r) - \left(\frac{r+3}{3}-\frac{3}{2(r+3)} \right)\delta |\mathcal{G}| - |\mathcal{G}|&= \sum_{x \in V(\mathcal{G})} |N^r(x)|-\left(\frac{r+3}{3}-\frac{3}{2(r+3)} \right)\delta |\mathcal{G}|\\
&\geq \frac{3}{2(r+3)}\delta |\mathcal{G}| +  \sum_{i=1} ^l \left(|X_i|^2-|X_i|\delta  \right) \\
&\geq \frac{1}{4}\delta^2 l + \sum_{i=1} ^l \left( |X_i|^2-|X_i|\delta  \right)\\
&= \sum_{i=1} ^l \left(|X_i|^2-|X_i|\delta +\frac{1}{4}\delta^2 \right)\\
&=\sum_{i=1} ^l \left(|X_i| - \frac{1}{2}\delta \right)^2  \\
&\geq 0 .\\ 
\end{align*}
The first equality uses the fact that $\mathcal{G}^r$ contains a loop at every vertex, hence $2e(\mathcal{G}^r)
=\sum_{x \in V(\mathcal{G})} |N^r(x)|+ |\mathcal{G}|$.  The first inequality follows from the definition of ``sufficent vertex", Claim \ref{lem:insufficientbound} and rearranging, while the second follows from Claim \ref{lem:vertexbound}.  This completes the proof.
\end{proof}

\begin{proof}[Proof of Theorem~\ref{thm:noloops}]
Let $\mathcal{G}$ be a copy of $G$ with a loop added at every vertex.  Then $\mathcal{G}^r$ will be isomorphic to  $G^r$ with a loop added at every vertex.  Note that we have $e(\mathcal{G}^r)= e(G^r)+ |G|$, and $\delta(\mathcal{G})=\delta(G)+1$. Substitute these into Theorem~\ref{thm:loops} obtain the following.
\begin {itemize}
 \item If $r\equiv 0 \pmod{3}$, then we have
$$e(G^r)\geq \left(\frac{r+3}{6}-\frac{3}{4(r+3)} \right) \delta(G) |G|  + \left(\frac{r+3}{6}-\frac{3}{4(r+3)} -\frac{1}{2} \right) |G|.$$
 \item If $r\not \equiv 0 \pmod{3}$, then we have
$$e(G^r)\geq \frac{1}{2}\left\lceil \frac{r}{3} \right\rceil \delta(G) |G| + \left(\frac{1}{2}\left\lceil \frac{r}{3} \right\rceil -\frac{1}{2} \right)|G|.$$
\end {itemize}

Note that for $r\geq 3$, both $\frac{r+3}{6}-\frac{3}{4(r+3)} -\frac{1}{2} $ and  $\frac{1}{2}\left\lceil \frac{r}{3} \right\rceil -\frac{1}{2} $ are non-negative, so Theorem~\ref{thm:noloops} follows.
\end{proof}

\bigskip\noindent
\textbf{Acknowledgment}

\smallskip\noindent
The author would like to thank his supervisors Jan van den Heuvel and Jozef Skokan for advice and discussions.

\bibliography{cdpref}
\bibliographystyle{abbrv}
\end{document}